\documentclass[12pt]{amsart}

\usepackage{amsmath,amssymb,amsbsy,amsfonts,amsthm,latexsym,mathabx,
            amsopn,amstext,amsxtra,euscript,amscd,stmaryrd,mathrsfs,
            array,mathtools,enumerate}
\usepackage{url}
\usepackage[colorlinks,linkcolor=blue,anchorcolor=blue,citecolor=blue,backref=page]{hyperref}
\usepackage{color}
\usepackage[utf8]{inputenc}
\usepackage[norefs,nocites]{refcheck}
\usepackage{mathtools}
\usepackage{todonotes}
\usepackage{float}
\hypersetup{breaklinks=true}
\usepackage[english]{babel}
\usepackage{enumitem}
\usepackage{a4wide}
\usepackage{enumerate}
\usepackage[noadjust]{cite}
\usepackage[curve]{xypic}

\def\le{\leqslant}
\def\ge{\geqslant}
\def\leq{\leqslant}
\def\geq{\geqslant}

\DeclareMathOperator{\Gal}{Gal}

\DeclareMathOperator{\ind}{ind}
\DeclareMathOperator{\dens}{dens}
\DeclareMathOperator{\rank}{rank}

\newtheorem{theorem}{Theorem}

\newtheorem{cor}[theorem]{Corollary}
\newtheorem{prop}[theorem]{Proposition}
\newtheorem{definition}{Definition}

\newtheorem{example}[theorem]{Example}
\newtheorem{rem}[theorem]{Remark}

\numberwithin{equation}{section}
\numberwithin{theorem}{section}
\numberwithin{table}{section}

\numberwithin{figure}{section}

\def\\{\cr}
\def\({\left(}
\def\){\right)}
\def\<{\langle}
\def\>{\rangle}

\def\le{\leqslant}
\def\ge{\geqslant}

\def \sB{{\mathscr B}}

\def\cO{{\mathcal O}}

\def\cS{{\mathcal S}}

\def\Q{\mathbb{Q}}

\title[Uniform bounds for the density in Artin's conjecture]{Uniform bounds for the density\\ in Artin's conjecture on primitive roots}
\author[A. Perucca and I.E. Shparlinski]{Antonella Perucca and Igor E. Shparlinski}

\address{A.P.: Department of Mathematics, University of Luxembourg, 
Esch-sur-Alzette,  L-4364, Luxembourg}
\email{antonella.perucca@uni.lu}

\address{I.S.: School of Mathematics and Statistics, University of New South Wales.
Sydney, NSW 2052, Australia}
\email{igor.shparlinski@unsw.edu.au}

\begin{document}

\begin{abstract}
We consider Artin's conjecture on primitive roots over a number field $K$, reducing an algebraic number $\alpha\in K^\times$. Under the Generalised Riemann Hypothesis, there is a density $\dens(\alpha)$ counting the proportion of the primes of $K$ for which $\alpha$ is a primitive root.
This density $\dens(\alpha)$ is a rational multiple of an Artin constant $A(\tau)$ that depends on the largest integer $\tau\geq 1$ such that 
$\alpha\in \(K^\times\)^\tau$. 
The aim of this paper is bounding the ratio $\dens(\alpha)/A(\tau)$, under the assumption that $\dens(\alpha)\neq 0$. Over $\mathbb Q$, this ratio is between $2/3$ and $2$, these bounds being optimal. For a general number field  $K$ we provide upper and lower bounds that only depend on $K$.
\end{abstract}

\keywords{primitive root, Artin's conjecture}
\subjclass[2010]{11A07, 11M26, 11R45}

\pagenumbering{arabic}

\maketitle

\section{Introduction}

We consider Artin's conjecture on primitive roots over number fields, pointing the reader to Moree's survey~\cite{MoreeArtin} for an extensive introduction to this topic 
and a very rich list of references. 
The classical conjecture concerns the field $\mathbb Q$ and the reductions of a rational number $\alpha \notin \{0,\pm 1\}$ which is not a square. Then, under the Generalised Riemann Hypothesis (GRH), the following set of prime numbers $p$ has a positive Dirichlet density $\dens(\alpha)$: the primes $p$ such that $\alpha$ is a primitive root modulo $p$, which means that $v_p(\alpha)=0$
 and that $(\alpha \bmod p)$ generates  $(\mathbb Z/p\mathbb Z)^\times$.
Letting $\tau\geq 1$ be the largest integer such that $\alpha$ is a  $\tau$-th power in $\mathbb Q^\times$, the heuristic density is the Artin constant 
$$
A(\tau) :=  \prod_{\ell\mid \tau}\left(1-{1\over \ell-1}\right)\prod_{\ell\nmid \tau}\left(1-{1\over \ell(\ell-1)}\right), 
$$
 where $\ell$ varies over the prime numbers.
It is known that $A(1)\sim 0.37$ and $A(\tau)\leq A(1)$ is a positive rational multiple of $A(1)$.
The density $\dens(\alpha)$ is a positive rational multiple of $A(\tau)$, and the  aim of this work is bounding the ratio ${\dens(\alpha)}/{A(\tau)}$.

If $\tau=1$, then we have $$\frac{84}{85}\leq \frac{\dens(\alpha)}{A(1)}\leq \frac{6}{5}$$ 
where the bounds are attained for  $\alpha=-15$ and $\alpha=-3$, respectively.
In general, we have the following bounds (attained for $\alpha=(-15)^{15}$ and $\alpha=(-3)^{3}$,  respectively): $$\frac{2}{3}\leq \frac{\dens(\alpha)}{A(\tau)} \leq 2\,.$$
We deduce that the largest Artin density over $\mathbb Q$ is 
$$
\dens(-3)=\frac{6}{5}A(1)\sim 0.45\, .
$$

We also consider  this  problem for a general number field $K$.  
We again assume the GRH (more precisely, we assume the Extended Riemann Hypothesis for the zeta function of number fields). We suppose that $\alpha\in K^\times$ is not a root of unity and it is not a square in $K^{\times}$, and we exclude the finitely many primes $\mathfrak p$ of $K$ such that 
$v_\mathfrak p(\alpha)\neq 0$, which ensures that $(\alpha \bmod \mathfrak p)$ is well-defined and non-zero. 
Then $\dens(\alpha)$ is the density of the primes $\mathfrak p$ of $K$ such that
$(\alpha \bmod \mathfrak p)$ generates the cyclic group $(\cO/\mathfrak p)^\times$, 
where $\cO$ denotes  the ring of integers of $K$.

For every positive integer $n$ let $\zeta_n = \exp(2\pi i /n)$.
Denote by $B$ the square-free part of the smallest even integer $\Omega\geq 1$ such that the largest abelian subextension of $K/\mathbb Q$ is contained in the cyclotomic field $\mathbb Q(\zeta_\Omega)$ (in particular, $B$  depends only on $K$). We restrict to the case $\dens(\alpha)\neq 0$ (see Lenstra's work~\cite{Lenstra} for a study of this condition) and in particular the largest integer $\tau$ such that $\alpha\in \(K^{\times}\)\tau$ is odd. By Theorem~\ref{upperbound} we have 
$$\frac{\dens(\alpha)}{A(\tau)}\leq \prod_{\ell \mid B, \ell\mid \tau}
\frac{\ell-1}{\ell-2}
\prod_{\ell \mid B, \ell\nmid \tau}
\frac{\ell^2-\ell}{\ell^2-\ell-1}
\leq 2 \prod_{\ell \mid (B/2)}
\frac{\ell-1}{\ell-2} \,.
$$
So we have an upper bound that only depends on $B$:  since $A(\tau)$ can be arbitrarily small by varying $\tau$, our uniform bound has an advantage with respect to the trivial bound $\dens(\alpha)\leq 1$.

Finally, we prove that there exists an explicit positive constant $c_B$ that only depends on $B$ (in particular, it only depends on $K$) such that 
$$\frac{\dens(\alpha)}{A(\tau)}\geq c_B\,.$$
Our results for $\mathbb Q$ rely on the explicit formulas for $\dens(\alpha)$ provided by Hooley~\cite{Hooley}.
For general $K$, our upper bound  stems from the following observation (where by `index' we mean the index of $(\alpha \bmod \mathfrak p)$, namely the index of the group $\langle(\alpha \bmod \mathfrak p)\rangle$ inside  $(\cO/\mathfrak p)^\times$): the set of primes $\mathfrak p$ of $K$ such that the index is $1$ is  contained in the 
set of primes $\mathfrak p$ of $K$ such that all prime divisors of the index divide $B$. Finally, our lower bound is based on insights on the so-called \emph{entanglements}, namely, the interdependencies among the cyclotomic-Kummer extensions  $K(\zeta_\ell, \sqrt[\ell]{\alpha})/K$, where $\ell$ varies in the set of prime numbers.

We may replace $\alpha$ by a finitely generated and torsion-free subgroup $G$ of $K^\times$ of  positive rank, and then compare the Artin density $\dens(G)$ to a suitable Artin constant $\mathcal A_{\mathcal R}$ (see Definition~\ref{defi}). We are able to provide upper and lower bounds for  $\dens(G)/\mathcal A_{\mathcal R}$ that only depend on $K$, see Corollary~\ref{cor:KK}.

Finally remark  that Section~\ref{secent} contains Kummer theory results describing the entanglements among the field extensions  $K(\zeta_\ell, \sqrt[\ell]{G})/K$, where $\ell$ varies in the set of prime numbers. Those  results are unconditional and are of independent interest.

\subsection*{Acknowledgements} The first author got curious about the fact that $-3$ maximizes the Artin density over the rational numbers. Then discussions between the two authors in 2022 at the Max-Planck-Institut für Mathematik in Bonn  and the University of Luxembourg led to the uniform bounds over $\mathbb Q$. The hospitality of both institutions is gratefully acknowledged.
Finally, the first author generalised the bounds to  number fields in 2024. 
We thank Hendrik Lenstra, Pieter Moree, Pietro Sgobba and Kate Stange for encouragement, and Olli J\"arviniemi for confirming the plausibility of our main result. Last but not least, we are grateful to Fritz Hörmann for providing the very general proof of Proposition~\ref{prop:Fritz}.

During the preparation of this work, I.S. was supported in part by  Australian Research Council Grants  DP230100530 and  DP230100534.

\section{The Artin  density for the rational numbers}

In this section we work over $\mathbb Q$, keeping the notation of the introduction and assuming the GRH. Remark that we use the notation $\ell$ to denote prime numbers, and $\mu$ is the Moebius function. 

Let $\alpha\in \Q\backslash\{0,\pm 1\}$ be not a square in $\mathbb Q^\times$, and call $\delta$ the square-free integer such that $\alpha/\delta$ is a square in $\mathbb Q^\times$.
Denote by $\tau$ the largest integer for which $\alpha \in \(\mathbb Q^{\times}\)^\tau$ and remark that $\tau$ is odd. 
By Hooley's work~\cite{Hooley}, the Artin density can then be written as 
\begin{equation}
\label{eq:Def dens}
\dens(\alpha) = A(\tau) \begin{cases} 1  & \text{if\ } \delta \not\equiv 1 \pmod 4,\\  
1-\mu(|\delta|)f_\tau(\delta)  &  \text{otherwise}\,,
\end{cases}
\end{equation}
where
$$
f_\tau(\delta)  = \prod_{\ell\mid \delta, ~\ell\mid \tau}\frac{1}{\ell-2}\prod_{\ell\mid \delta, ~\ell\nmid \tau}\frac{1}{\ell^2-\ell-1}\,.
$$
As the above quantities $\mu(|\delta|)$ and $f_\tau(\delta)$ are only needed for the case $\delta \equiv 1 (\bmod 4)$ we may define $\delta$, as done by Moree in~\cite{MoreeArtin}, as the discriminant of $\mathbb Q(\sqrt{\alpha})$.

\begin{theorem} \label{thm:boundA}
For $\tau=1$, we have $\frac{\dens(\alpha)}{A(1)}\leq \frac{6}{5}$ and 
for arbitrary $\tau$ 
$$ 2 \ge  \frac{\dens(\alpha)}{A(\tau)} \geqslant  \begin{cases} \frac{84}{85}&\text{if $\tau=1,$}\\
\frac{18}{19} &  \text{if $\tau$ is power of $3$,}\\
\frac{2}{3} &\text{if $3\mid \tau$ and $t  = 5$,}\\
\frac{14}{15} &\text{if $3\nmid \tau$ and $t = 5$,}\\
1-\frac{1}{\min\{19, t-2\}} &\text{if $3\mid \tau$ and $t \ge 7$,}\\
1-\frac{1}{5 \min\{19, t-2\}} &\text{if $3\nmid \tau$ and $t \ge 7$,}\\
\end{cases}
$$
where $t$ is the smallest prime which exceeds $3$ and divides $\tau$, if it exists, and $t=0$ otherwise, and each of the lower and upper bounds is attained. 
\end{theorem}
\begin{proof} 
We rely on the explicit formula~\eqref{eq:Def dens}. If $\delta\not \equiv 1 (\bmod 4)$ we have ${\dens(\alpha)}/{A(\tau)}=1$, so we may restrict to the case $\delta\equiv 1 (\bmod 4)$. We  need to maximize 

$$f_\tau(\delta)=\prod_{\ell\mid \delta, ~\ell\mid \tau} F(\ell) \prod_{\ell\mid \delta, ~\ell\nmid \tau} G(\ell)\,,$$
where
$$F(\ell) := \frac{1}{\ell-2}\quad,\quad G(\ell) := \frac{1}{\ell^2-\ell-1}\,.$$
To get an upper (respectively, lower) bound we have to choose $\delta$ with negative (respectively, positive) Moebius function $\mu(|\delta|)$.
We can take $|\delta|$ prime (respectively, the product of two primes) because $f_\tau(\delta)$ does not increase by adding prime factors to $\delta$.
Remark that
  \begin{equation}
\label{eq:FG val}
 \begin{split}
 &F(3) =1 >  G(3)=1/5 \ge \max\{F(\ell), G(\ell):~ \ell \ge 5, \ \text{prime}\};\\
& F(5) =1/3 > G(5)=1/19\ge \max\{F(\ell), G(\ell):~ \ell \ge 23, \ \text{prime}\}.
 \end{split}
 \end{equation}
We have in particular 
 $$
f_1(\delta) = \prod_{\ell\mid \delta}G(\ell) \le  \min_{\ell\mid \delta}G(\ell) \le G(3) = \frac{1}{5} = f_1(-3)\,.
$$
So, supposing $\tau=1$, for $\alpha=-3$ we get the largest value ${\dens(\alpha)}/{A(1)}=6/5$.
Now consider a general  $\tau\geq 1$. We have 
$$
 f_\tau(\delta)  = \prod_{\ell\mid \delta, ~\ell\mid \tau}F(\ell)
 \prod_{\ell\mid \delta, ~\ell\nmid \tau}G(\ell)  \ll
 \max\left\{\max_{\ell\mid \delta, ~\ell\mid \tau} F(\ell), 
 \max_{\ell\mid \delta, ~\ell\nmid \tau} G(\ell)\right\} .
$$
Thus, considering~\eqref{eq:FG val}, we get  
$ f_\tau(\delta)  \leq    f_3(-3)=1$.
So for $\alpha=(-3)^3$ we get the largest  value ${\dens(\alpha)}/{A(\tau)}=2$.

To get a lower bound,  we have to find the optimal choice for two distinct odd primes $\ell_1,\ell_2$ and take $\delta \in \{ \pm \ell_1\ell_2\}$, recalling that $\delta\equiv 1 \bmod 4$. By~\eqref{eq:FG val} we can take  $\ell_1 = 3$.
In the first four cases of the lower bound we may take $\ell_2= 5$ and hence $\delta=-15$, so we easily conclude those cases by computing ${\dens((-15)^\tau)}/{A(\tau)}$.

In the penultimate case (where $5\nmid \tau$) by~\eqref{eq:FG val} we take $\ell_2 = 5$ if $t\ge 23$ and $\ell_2 = t$ otherwise
  and obtain 
  $$
f_\tau(\delta)   = F(3) \max\{G(5), F(t)\}  = \frac{1}{\min\{19, t-2\}}\,.
$$
In the last case, a similar argument leads to the same choice of $\delta$ and hence 
$$
f_\tau(\delta)  = G(3) \max\{G(5), F(t)\}  = \frac{1}{5\min\{19, t-2\}}\,,
$$
which concludes the proof. 
\end{proof}

\begin{example} By Hooley's formula~\eqref{eq:Def dens}
 we clearly obtain the largest value for $\dens(\alpha)$ with $\tau=1$. Then the largest Artin density over $\mathbb Q$ is $\dens(-3)=\frac{6}{5}A(1)$. In fact, $-3$ maximizes the Artin density because if the index of $(-3 \bmod p)$ is odd, then it is automatically coprime to $3$ (indeed, we are considering primes $p$ that do not split completely in $\mathbb Q(\sqrt{-3})=\mathbb Q(\zeta_3)$).
\end{example}

\section{Setup for the general case}

\subsection{Notation} \emph{We  assume the Extended Riemann Hypothesis for the zeta function of number fields.}

Let $K$ be a number field, and work within a fixed algebraic closure of $K$. Write $\zeta_n$ for a primitive $n$-th root of unity. 

Let $\alpha\in K^\times$ be not a root of unity, while $G$ is a finitely generated and torsion-free subgroup of $K^\times$ of positive rank $r$. While considering the reductions of $\alpha$ (respectively, $G$) modulo primes  of $K$ we tacitly exclude the finitely many primes $\mathfrak p$ such that the reduction modulo $\mathfrak p$ is not well-defined or it is not contained in the multiplicative group of the residue field at $\mathfrak p$. Thanks to this restriction we have a well-defined multiplicative index of $\alpha$ modulo $\mathfrak p$ (respectively, of $G$  modulo $\mathfrak p$) in the multiplicative group of the residue field at $\mathfrak p$.

We call $\dens(\alpha)$ the density of  primes $\mathfrak p$ of $K$ for which $\ind(\alpha \bmod \mathfrak p)=1$. We similarly define $\dens(G)$ requiring $\ind(G \bmod \mathfrak p)=1$. 

\subsection{Artin constants}
By the very general framework developed by the first author with J\"arviniemi and Sgobba~\cite{JP, JPS}, we know in particular that the Artin density $\dens(G)$ is a rational multiple of the following Artin constant (which is at least $A(1)$ and it is strictly less than $1$):   $$\prod_{\ell \text{  prime}}\left(1-{1\over \ell^r(\ell-1)}\right)\,.$$
We now define a  rational multiple of the above constant with the aim of producing an Artin constant that better fits $\dens(G)$: 
\begin{definition}\label{defi}
For every prime $\ell$ let $r_\ell $ be an integer in the range from $0$ to $r$, such that $r_\ell=r$ holds for all but finitely many $\ell$. Then we define the Artin constant
$$A_{\mathcal R}:= \prod_{\ell \text{  prime}}\left(1-{1\over \ell^{r_\ell}(\ell-1)}\right)\,.$$
\end{definition}

Given $G$ as above, for every prime $\ell$ we define $r_\ell$ as the rank of the group  $G/(G\cap K^{\times \ell})$.
Then we have $A_{\mathcal R}>0$ if we suppose that $r_2>0$ (in fact, $r_2=0$ implies that $G$ consists of squares and hence $\dens(G)=0$).

\begin{example}
For $G=\langle \alpha \rangle$ we have $r_\ell\in \{0,1\}$. Suppose  that $\alpha$ is not a square in $K^\times$, and call $\tau$ the largest (odd) integer such that 
$\alpha\in \(K^\times\)^\tau$. Then we have $r_\ell=0$ if and only if $\ell\mid \tau$, which implies that $A_{\mathcal R}=A(\tau)$.
\end{example}

\section{Entanglements}\label{secent}

\emph{The results in this section are unconditional.}
We denote by $Q$ the product of the prime numbers $q$ such that $\zeta_q\in K$. Moreover, we denote by $B$ the square-free part of the smallest even integer $\Omega\geq 1$ such that the largest abelian subextension of $K/\mathbb Q$ is contained in $\mathbb Q(\zeta_\Omega)$. In particular, we have $Q\mid B$. Remark that $Q$ and $B$ are constants that only depend on $K$ (more precisely, they only depend on the largest abelian subextension of $K/\mathbb Q$). We remark that for the results in this section it is possible to replace $B$ by any positive multiple of it which is again square-free.

\begin{rem}\label{helpful}
Because of our choice of $B$ the following holds.
For all primes $\ell\nmid B$ the cyclotomic fields $K(\zeta_\ell)$ are linearly disjoint and have maximal degree $\ell-1$ over $K$. Moreover, for any fixed prime $\ell\nmid B$, and denoting any prime by $\widetilde \ell$, we have
$$ K(\zeta_\ell) \cap \prod_{\widetilde \ell \neq \ell}K(\zeta_{\widetilde \ell})=K\,.$$
Consequently, we have 
$$ K(\zeta_B) \cap \prod_{\ell \nmid B}K(\zeta_{\ell})=K\,.$$
\end{rem}

\begin{prop}\label{helpful2}
For every prime $\ell\nmid B$ the degree of the cyclotomic-Kummer extension $K(\zeta_\ell, \sqrt[\ell]{G})/K$ equals $\ell^{r_\ell}(\ell-1)$. For every prime $\ell$ such degree divides $\ell^{r_\ell}(\ell-1)$. \end{prop}
\begin{proof}
The second assertion is straight-forward by classical Kummer theory. 
By Remark~\ref{helpful} we are left to prove that for $\ell\nmid B$ the degree of $K(\zeta_\ell, \sqrt[\ell]{G})/K(\zeta_\ell)$
equals $\ell^{r_\ell}$. 
We claim that $r_\ell$ is the rank of $G/(G\cap K(\zeta_\ell)^{\times \ell})$, so the assertion follows from classical Kummer theory. The claim holds  because $\ell\neq 2$ so with the language of~\cite{Debry} the  parameters for $\ell$-divisibility of $G$ are the same over $K$ and over $K(\zeta_\ell)$.
\end{proof}

\begin{prop}\label{helpful3}
For every prime $\ell\nmid B$ we have
$K(\zeta_\ell, \sqrt[\ell]{G})\cap K(\zeta_\infty)=K(\zeta_\ell)$. Moreover, (for any choice of the $\ell$-th roots) we have $K( \sqrt[\ell]{G})\cap K(\zeta_\infty)=K$. 
For all primes $\ell\nmid B$ the cyclotomic-Kummer extensions  $K(\zeta_\ell,  \sqrt[\ell]{G})/K$ are linearly disjoint over $K$.
\end{prop}
\begin{proof}
The field $K(\zeta_\ell,  \sqrt[\ell]{G})\cap K(\zeta_\infty)$ is an abelian extension of $K$ so by Schinzel's  result on abelian radical extensions~\cite[Theorem~2]{Schinzel} it must be contained in $K(\zeta_\ell)$ because $\ell\nmid Q$. The second assertion follows because the extensions $K(\zeta_\ell)/K$ and $K(\sqrt[\ell]{G})/K$ have coprime degrees.

For the last assertion it  suffices to observe that for every prime $\ell\nmid B$ and by varying $\widetilde \ell$ in the prime numbers we have $$K(\zeta_\ell, \sqrt[\ell]{G})\cap \prod_{\widetilde \ell \nmid  \ell B} K(\zeta_{\widetilde \ell}, \sqrt[\widetilde \ell]{G})\subseteq K(\zeta_\ell)\cap \prod_{\widetilde \ell\nmid \ell B} K(\zeta_{\widetilde \ell}, \sqrt[\widetilde \ell]{G}) = K(\zeta_\ell)\cap \prod_{\widetilde \ell\nmid \ell B} K(\zeta_{\widetilde \ell})=K\,.$$
The inclusion is because $K(\sqrt[\ell]{G})/K$ and  $K(\sqrt[\widetilde \ell]{G})/K$ have coprime degrees and then we may apply the first part of the statement.
The former equality again follows from the first part of the statement, while the latter equality is a consequence of Remark~\ref{helpful}.
\end{proof}

Let $B_G$ be the squarefree integer which is the smallest positive multiple of $B$ such that for any fixed prime $\ell\nmid B_G$ and denoting any prime by $\widetilde \ell$ we have
$$ K(\zeta_\ell, \sqrt[\ell]{G}) \cap \prod_{\widetilde \ell \neq \ell}K(\zeta_{\widetilde \ell}, \sqrt[\widetilde \ell]{G}))=K$$
(the existence of $B_G$ follows from Kummer theory, see~\cite[Section~3]{JP} by the first author and J\"arviniemi). We remark that in the following results we could replace $B_G$ by a positive multiple of it which is again square-free.

A proof of our next result has been communicated to the authors by Fritz Hörmann.

\begin{prop}\label{prop:Fritz}
Let $K$ be a field, let $L_1, L_2, M$ be finite abelian extensions of $K$ and consider the field $E:=ML_1 \cap L_2$. Assuming that $L_1 \cap L_2 = K$, the Galois group 
$\Gal(E/K)$ of $E/K$ is a quotient of a subgroup of the Galois group $\Gal(M/K)$ of $M/K$.
\end{prop}
\begin{proof}
The compositum $F:=L_1L_2M$  is again a finite abelian extension of $K$. 
We write $G:=\Gal(F/K)$ and for $X\in \{L_1,L_2, M\}$ we write $G_X:=\Gal(F/X)$.
By our assumption, we have $G=G_{L_1} G_{L_2}$. 
First observe that we have 
\begin{equation}\label{FH}
\Gal(E/K)\cong \frac{G_{L_1}G_{L_2}}{G_{L_2}(G_{L_1} \cap G_M)}\cong \frac{G_{L_1}}{(G_{L_1} \cap G_M)(G_{L_1} \cap G_{L_2})}\,,
\end{equation}
where the second isomorphism is obtained by observing that $G_{L_2} \cap (G_{L} \cap G_M) = 1$ (by the definition of $F$) and hence
\[ \frac{G_{L_1}G_{L_2}}{G_{L_2}(G_{L_1} \cap G_M)} \cong \frac{\frac{G_{L_1}G_{L_2}}{G_{L_2}}}{G_{L_1} \cap G_M} \cong  \frac{\frac{G_{L_1}}{G_{L_1} \cap G_{L_2}}}{G_{L_1} \cap G_M} \cong  \frac{G_{L_1}}{(G_{L_1} \cap G_M)(G_{L_1} \cap G_{L_2})}\,.   \]
The last group in~\eqref{FH} is a quotient of 
$$\frac{G_{L_1} }{G_{L_1} \cap G_M}\cong \frac{G_{L_1}  G_M}{G_M}\,.$$
Finally, this last group is a subgroup of
$$\frac{G_{L_1} G_{L_2}}{G_M}\cong \Gal(M/K)\,$$
and the result follows. 
\end{proof}

We recall that the exponent of an abelian extension is defined as the 
exponent of its Galois group.

We instantly derive from Proposition~\ref{prop:Fritz} one of our main tools. 

\begin{cor}\label{cor:algebraic}
Let $L_1$, $L_2$ and $M$ be finite abelian extensions of $K$ such that $L_1\cap L_2=K$. Then the field
$M L_1 \cap L_2 \subseteq L_2$ is an abelian extension of $K$ of degree dividing $[M:K]$ and of exponent dividing the exponent of $M/K$.
\end{cor}

We also need the following result. 

\begin{prop}\label{helpful4}
For every prime $\ell$ let $G_\ell$ be a subgroup of $G$.
Consider the field
$$F:=\prod_{\ell\mid  B} K(\zeta_\ell, \sqrt[\ell]{G_\ell})\cap \prod_{\ell\nmid  B} K(\zeta_\ell, \sqrt[\ell]{G_\ell})\,.$$
Then we have 
$$F\subseteq \Bigl(K(\zeta_B)\prod_{q\mid Q}K(\sqrt[q]{G_q})\Bigr) \cap K(\zeta_{B_G/B})$$
and $F/K$ is a Kummer extension of exponent dividing $Q$ and degree dividing $Q^{r_Q}$, where $r_Q:=\max_{q\mid Q} \rank G_q$.
\end{prop}

\begin{proof}
The second assertion follows from the first as a consequence of 
Corollary~\ref{cor:algebraic} (setting $M=\prod_{q\mid Q}K(\sqrt[q]{G_q})$,  $L_1=K(\zeta_B)$ and $L_2=K(\zeta_{B_G/B})$) because $M/K$ is a Kummer extension of exponent dividing $Q$ and degree dividing $Q^{r_Q}$.
Also remark that by the definition of $B_G$ we have
$$F=\prod_{\ell\mid  B} K(\zeta_\ell, \sqrt[\ell]{G_\ell})\cap \prod_{\ell\mid (B_G/B)} K(\zeta_\ell, \sqrt[\ell]{G_\ell})\,.$$

For any prime $\widetilde \ell$, we consider the ${\widetilde \ell}$-part $F_{\widetilde \ell}$ of $F$, namely the largest subextension of $F$ whose degree is a power of ${\widetilde \ell}$.
For a prime $q\mid Q$ we clearly have 
$$F_{q} \subseteq \Bigl(K(\sqrt[q]{G_q})\prod_{\ell\mid  (B/q)} K(\zeta_\ell)\Bigr)\cap \prod_{\ell\mid  (B_G/B)} K(\zeta_\ell)\,.$$
So we conclude by proving that for $\widetilde \ell\nmid Q$ we have $F_{\widetilde \ell}=K$.

For $\widetilde \ell\nmid B$ we have
$$F_{\widetilde \ell} \subseteq \prod_{\ell\mid  B} K(\zeta_\ell)\cap \Bigl(K(\zeta_{\widetilde \ell}, \sqrt[\widetilde \ell]{G_{\widetilde \ell}})\prod_{\ell\mid  (B_G/B),\,\ell\neq \widetilde \ell} K(\zeta_\ell)\Bigr)\,.$$
Considering the former field on the right-hand-side, $F_{\widetilde \ell}/K$ is abelian. 
So we may replace the latter field by its largest abelian subextension. Since $\zeta_{\widetilde \ell}\notin K$, by Schinzel's Theorem on abelian radical extensions~\cite[Theorem~2]{Schinzel} we then have
$$F_{\widetilde \ell} \subseteq \prod_{\ell\mid  B} K(\zeta_\ell)\cap \prod_{\ell\mid  (\widetilde \ell B_G/B)} K(\zeta_\ell)=K\,,$$
where the equality holds by Remark~\ref{helpful}.

For $\widetilde \ell\mid B$ with $\widetilde \ell\nmid Q$ we may reason analogously because we have
$$F_{\widetilde \ell} \subseteq \Bigl(K(\zeta_{\widetilde \ell}, \sqrt[\widetilde \ell]{G_{\widetilde \ell}})\prod_{\ell\mid  (B/\widetilde \ell)} K(\zeta_\ell)\Bigr)\cap \prod_{\ell\mid (B_G/ B)} K(\zeta_\ell)\,,$$
which concludes the proof. 
\end{proof}

As remarked above, $B$ and $B_G$ in Proposition~\ref{helpful4} can be replaced by positive square-free multiples that still satisfy $B\mid B_G$.

\begin{rem}
As a special case of Proposition~\ref{helpful4}
we have
$$\prod_{\ell\mid  B}K(\zeta_\ell, \sqrt[\ell]{G})\cap \prod_{\ell\nmid  B} K(\zeta_\ell, \sqrt[\ell]{G})\subseteq K(\zeta_B, \sqrt[Q]{G}) \cap K(\zeta_{B_G/B}).$$
\end{rem}

\section{General bounds for the Artin density}

All densities mentioned in this section exist thanks to the results in~\cite{JP} by the first author and J\"arviniemi (which are conditional under the Extended Riemann Hypothesis).
Let $B$ be as   in  Section~\ref{secent}. We  denote by 
$\dens_B(G)$ the density of the primes $\mathfrak p$ of $K$ such that for every prime divisor 
$\ell\mid B$ we have $\ell\nmid\ind(G\bmod \mathfrak p)$.
For any prime $\ell$,  we write  $\dens_\ell(G)$ for the density of the primes 
$\mathfrak p$ of $K$ such that $\ell\nmid\ind(G\bmod \mathfrak p)$.
Finally, we denote by $\widetilde \dens_B(G)$ the density of the primes $\mathfrak p$ of $K$ such that for every $\ell\nmid B$ we have $\ell\nmid\ind(G\bmod \mathfrak p)$. 

\begin{theorem}\label{upperbound}
Suppose that $\dens(G)\neq 0$ (in particular, $r_2>0$). Then we have
$$\frac{\dens(G)}{A_{\mathcal R}}\leq \prod_{\ell\mid B}\left(1-{1\over \ell^{r_\ell}(\ell-1)}\right)^{-1}\leq 2 \prod_{\ell\mid (B/2)} \frac{\ell-1}{\ell-2}
\,.$$
\end{theorem}

\begin{proof}
The second inequality is clear because we can estimate the factor $\ell=2$ using $r_2=1$ and each factor $\ell\neq 2$ using $r_\ell=0$.

For every prime $\ell\nmid B$ the degree of $K(\zeta_\ell, \sqrt[\ell]{G})/K$ equals $\ell^{r_\ell}(\ell-1)$ by  Proposition~\ref{helpful2}.
Moreover, such fields are  linearly disjoint by Proposition~\ref{helpful3}.
Remark that by Hooley's method  (see Moree's survey~\cite{MoreeArtin}) the condition  $\ell\nmid \ind(G\bmod \mathfrak p)$ can be replaced by $\mathfrak p$ not splitting completely in $K(\zeta_\ell, \sqrt[\ell]{G})$. We deduce that 
$$
\widetilde \dens_B(G)=\prod_{\ell\nmid B} \dens_\ell(G)=\prod_{\ell\nmid B}\left(1-{1\over \ell^{r_\ell}(\ell-1)}\right)\,.
$$
Since 
$$\frac{\dens(G)}{A_{\mathcal R}}\leq \frac{\widetilde \dens_B(G)}{A_{\mathcal R}}=\prod_{\ell\mid B}\left(1-{1\over \ell^{r_\ell}(\ell-1)}\right)^{-1}\,,$$
we conclude the proof. 
\end{proof}

We recall the definition of $B_G$ from  Section~\ref{secent}.

\begin{theorem}\label{lowerbound}
There exists a positive constant $c_{B}$ that depends only on $B$ such that if 
$\dens(G)\neq 0$ then we have
$$\frac{\dens(G)}{A_{\mathcal R}} \ge c_{B} \,.$$
\end{theorem}
\begin{proof}
We replace $B$ by the smallest square-free positive multiple $\sB$ so that for every prime $\ell\nmid B$ the integer $\ell-1$ is at least $2Q$,  where, as before, $Q$ denotes the product of the prime numbers $q$ such that $\zeta_q\in K$.
Observing that $Q\leq B$ we see that  $\sB$ only depends on the original value for $B$. 
We call $\sB_G$ the least common multiple between $\sB$ and $B_G$ and remark that $\sB_G$ is squarefree. We point out that our proof strategy involves concentrating on the prime divisors of $\sB_G$, distinguishing the primes that divide $\sB$ by those who don't.

By the independence and maximality of the cyclotomic-Kummer extensions at primes not dividing $\sB_G$ (as described in  Section~\ref{secent}) we have 
$$\dens(G)=\dens_{\sB_G}(G) \prod_{\ell\nmid \sB_G} \dens_\ell(G)= \dens_{\sB_G}(G) \prod_{\ell\nmid \sB_G} \left(1-{1\over \ell^{r_\ell}(\ell-1)}\right)\,.$$
Thus we have
 $$\frac{\dens(G)}{A_{\mathcal R}}=\frac{\dens_{\sB_G}(G)}{\prod_{\ell\mid \sB_G} \left(1-{1\over \ell^{r_\ell}(\ell-1)}\right)}
 \geq \frac{\dens_{\sB_G}(G)}{\prod_{\ell\mid (\sB_G/\sB)} \left(1-{1\over \ell^{r_\ell}(\ell-1)}\right)}\,.$$
 Since $\dens(G)\neq 0$ we must have $\dens_{\sB}(G)\neq 0$ so there exists a Galois automorphism $\tau$ on
 $$\prod_{\ell\mid \sB}K(\zeta_\ell, \sqrt[\ell]{G}),$$
 which is not the identity on any of the fields $K(\zeta_\ell, \sqrt[\ell]{G})$. We can fix a set of generators of $G$ and for any $\ell\mid \sB$ we choose one of those generators, call it $\alpha_\ell$, which is not fixed by $\tau$, setting $\alpha_\ell=1$ if no such generator exists. We then call $\tau\mid_L$ the restriction of $\tau$ to the field
 $$L:=\prod_{\ell\mid \sB}K(\zeta_\ell, \sqrt[\ell]{\alpha_\ell})\,.$$
 For every positive integer $N$ we write $K_{N}:=\prod_{\ell\mid N}K(\zeta_\ell, \sqrt[\ell]{G})$.
 We then define $f_{\sB_G, \,\tau}(G)$ as the number of automorphisms in $\Gal(K_{\sB_G}/K)$
whose restriction to $L$ is $\tau\mid_L$, and such that for every $\ell\mid (\sB_G/\sB)$ the restriction to  $K(\zeta_\ell, \sqrt[\ell]{G})$ is not the identity.
By our choice of $\tau$ we clearly have
$$\frac{f_{\sB_G, \,\tau}(G)}{[K_{\sB_G}:K]} \leq \dens_{\sB_G}(G)$$
so we are left to prove that there is a constant $c^\star_B$ (that depends only on $B$) such that
\begin{equation}\label{eq: goal}
\frac{f_{\sB_G, \,\tau}(G)}{[K_{\sB_G}:K]} \ge c^\star_B
 \prod_{\ell\mid (\sB_G/ \sB)} \left(1-{1\over \ell^{r_\ell}(\ell-1)}\right)\,.
 \end{equation}
We consider the field
 $$E:=L\cap K_{\sB_G/\sB}\,.$$
We can write  
\begin{equation}\label{eq: rat1 * rat2}
\frac{f_{\sB_G, \,\tau}(G)}{[K_{\sB_G}:K]}  = \frac{1}{[L:E]}\cdot  \frac{f_{\sB_G, \,\tau\mid E}(G)}{[K_{\sB_G/\sB}:K]},
\end{equation}
 where $f_{\sB_G, \,\tau\mid E}(G)$ counts the automorphisms in $\Gal(K_{\sB_G/\sB}/K)$ whose restriction to $E$ equals $\tau\mid_E$ and such that for every $\ell\mid (\sB_G/ \sB)$ the restriction to  $K(\zeta_\ell, \sqrt[\ell]{G})$ is not the identity. Indeed, once we have fixed a suitable automorphism in $\Gal(K_{\sB_G/\sB}/K)$, then we only need to extend it in a prescribed way to  $LK_{\sB_G/\sB}$, and we have $[LK_{\sB_G/\sB}:K_{\sB_G/\sB}]= [L:E]$.
 Since $[L:E]\leq [L:K]\leq \prod_{\ell\mid \sB} \ell(\ell-1)$, the inverse of $[L:E]$ can be bounded from below only in terms of $B$.
 
 If $\sB_G=\sB$, then the last ratio in~\eqref{eq: rat1 * rat2} is $1$ and we may conclude, so in what
  follows we will suppose that $\sB_G$ has more prime divisors with 
respect to $\sB$.
 
Recall that by our choice of $\sB$ we have 
$$
 [K_{\sB_G/\sB}:K]=\prod_{{\ell\mid (\sB_G/\sB)} } \frac{1}{\ell^{r_\ell}(\ell-1)}
$$ 
 so, by~\eqref{eq: goal},  we are left to prove that there exists a constant $ c^\sharp_B$ depending only on $B$ such that
\begin{equation}\label{eq:Suff Cond}
f_{\sB_G, \,\tau\mid E}(G)\geq c^\sharp_B \prod_{\ell\mid (\sB_G/\sB)} \left(\ell^{r_\ell}(\ell-1)-1\right)\,.
\end{equation} 
By Proposition~\ref{helpful4} (setting $G_\ell=\langle \alpha_\ell \rangle$ for $\ell\mid \sB$ and $G_\ell=G$ otherwise, so in particular $r_Q=1$) we deduce that $E/K$ is a Kummer extension of degree dividing $Q$ contained in $K(\zeta_{\sB_G/\sB})$.

Fix a prime $q$ dividing $[E:K]$, so in particular we have $q\mid Q$. The $q$-part of the extension $E/K$ is a cyclic extension $C_q/K$ of degree $q$ contained in 
 $\prod_{\ell\in \cS_q}K(\zeta_\ell)$, 
 where $\cS_q$ is a subset of the primes dividing $\sB_G$, not dividing $\sB$, and such that $q\mid [K(\zeta_\ell):K]$. We may assume that $\cS_q$ is minimal.
 
Then (on the $q$-part of the considered Galois groups) the condition of extending the automorphism $\tau\mid_{C_q}$ is a compatibility of $q$-characters related to the subextensions of degree $q$ of $K(\zeta_\ell)$ for $\ell\in \cS_q$. And this condition is satisfied if we fix $\ell_q\in \cS_q$ and select an appropriate automorphism on the degree $q$ subextension of $K(\zeta_{\ell_q})/K$, such choice depending on the free choices that we can make for $K(\zeta_{\ell})/K$ by varying $\ell\in \cS_q$ with $\ell\neq \ell_q$.
 
 For every $\ell\mid (\sB_G/\sB)$ call $Q_\ell$ the product of the primes $q$ as above such that $\ell=\ell_q$ (setting $Q_\ell=1$ for an empty product). 
 So, up to considering fewer automorphisms, we are able to separate the conditions at the different primes, namely
\begin{equation}\label{eq: f vs g}
f_{\sB_G, \,\tau\mid E}(G)\geq \prod_{\ell\mid (\sB_G/\sB)} g_{r_\ell,Q_\ell},
\end{equation}
 where $g_{r_\ell,Q_\ell}$ is the amount of automorphisms in $K(\zeta_\ell, \sqrt[\ell]{G})$ that are not the identity and that are the fewest by varying a prescribed restriction to the cyclic subextension of degree $Q_\ell$ of $K(\zeta_\ell)$.
 We have (recalling that $\ell-1\geq 2Q \geq 2Q_\ell$)
\begin{equation}\label{eq: g  - fin}
g_{r_\ell,Q_\ell}\geq \frac{[K(\zeta_\ell, \sqrt[\ell]{G}):K]}{Q_\ell} -1=\frac{\ell^{r_\ell}(\ell-1)}{Q_\ell}-1\geq \frac{1}{ \min\{2, Q_\ell\}Q_\ell} (\ell^{r_\ell}(\ell-1)-1)\,.
\end{equation}
Remark that $$
 \prod_{\ell\mid (\sB_G/\sB)} Q_\ell \mid Q
 $$
 and that the number of primes $\ell\mid (\sB_G/\sB)$ such that $Q_\ell\neq 1$ is bounded by the number of prime divisors of $Q$, which is some constant $C_K$ that only depends on the field $K$.
Substituting~\eqref{eq: g  - fin} in~\eqref{eq: f vs g} we then obtain
$$
f_{\sB_G, \,\tau\mid E}(G)\geq \frac{1}{2^{C_K}Q} \prod_{\ell\mid (\sB_G/\sB)}  {\ell^{r_\ell}(\ell-1)-1}, 
$$
which  implies~\eqref{eq:Suff Cond} and concludes the proof.
 \end{proof}

Finally, since $B$ depends only on $K$, we observe that Theorems~\ref{upperbound}  and~\ref{lowerbound} imply that $\dens (G)/A_{\mathcal R}$ can be bounded only in terms of $K$,  which gives an analogues of Theorem~\ref{thm:boundA} for arbitrary number fields $K$ and finitely generated groups of $K^\times$, albeit less explicit.

\begin{cor}\label{cor:KK}
There exist two positive constants $c_0(K)$ and $C_0(K)$ that depend only on $K$  such that if 
$\dens(G)\neq 0$ then we have
$$c_0(K) \le \frac{\dens(G)}{A_{\mathcal R}} \le C_0(K)\,.$$
\end{cor}

\end{document}